\documentclass{amsart}

\usepackage{amsmath,amsthm,indentfirst}
\usepackage{amssymb}
\usepackage{amsfonts}
\usepackage{calrsfs}
\usepackage{amscd}
\usepackage[latin1]{inputenc}
\DeclareMathAlphabet{\mathpzc}{OT1}{pzc}{m}{it}
\theoremstyle{plain}
\newtheorem*{maintheorem*}{Main Theorem}
\newtheorem*{thm*}{Theorem}
\newtheorem*{thma*}{Theorem A}
\newtheorem*{thmaa*}{Theorem A'}
\newtheorem*{thmb*}{Theorem B}
\newtheorem*{thmo*}{Theorem 1.1}
\newtheorem*{thmc*}{Theorem C}
\newtheorem*{thmd*}{Theorem D}
\newtheorem*{thmf*}{Theorem 4.1}
\newtheorem*{remark*}{Remark}
\newtheorem*{conjecture*}{Conjecture}
\newtheorem*{prop*}{Proposition}
\newtheorem*{lem*}{Basic Lemma}
\newtheorem{thm}{Theorem}[section]
\newtheorem{cor}[thm]{Corollary}
\newtheorem{lem}[thm]{Lemma}
\newtheorem{prop}[thm]{Proposition}

\theoremstyle{definition}

\newtheorem*{proofc*}{Proof of Theorem C}

\newtheorem{definition}[thm]{Definition}

\def\o{\mathcal{O}}
\def\bbz{\mathbb{Z}}

\def\bbf{\mathbb{F}}
\def\bbr{\mathbb{R}}

\def\bbn{\mathbb{N}}

\def\bfr{\mathfrak{B}}
\def\bfrak{\mathfrak{b}}

\def\ybf{\mathbf{y}}

\def\wbf{\mathbf{w}}
\def\vbf{\mathbf{v}}
\def\vare{\varepsilon}

\def\u{\mathcal{U}}
\def\ubf{\mathbf{u}}

\def\S{\mathfrak{S}}

\def\imp{\mbox{Im}(\phi)}

\def\char{\rm{Char}}

\def\SL{\rm{SL}}

\def\vt{\vartheta}
\def\h{\hspace{1mm}}
\def\hh{\hspace{.5mm}}

\title[Isotropic Quadratic Forms In Positive Characteristic]{\sc Unipotent Flows And Isotropic Quadratic Forms In Positive Characteristic}
\author{A.~Mohammadi}
\date{}

\address{Mathematics Dept., University of Chicago, Chicago, IL}
\email{amirmo@math.uchicago.edu}

\begin{document}
\maketitle
\begin{abstract}

\noindent
The analogous statement to Oppenheim conjecture over a local field of positive characteristic is proved. The dynamical argument is most involved in the case of characteristic 3. 

\end{abstract}


\section{Introduction} 
Let $Q$ be a real non-degenerate indefinite quadratic form in $n$ variables. Further assume $Q$ is not proportional to a form with rational coefficients . It was conjectured by Oppenheim in~\cite{O2} that if $n\geq5$ then for any $\vare>0$ there is $x\in\bbz^n-\{0\}$ such that $|Q(x)|<\vare.$ Later on in 1946 Davenport stated the conjecture for $n\geq3,$ see~\cite{D1}. However the conjecture even in the case $n\geq3$ is usually referred to as Oppenheim conjecture. If the conjecture is proved for $n_0$ then it holds for $n\geq n_0.$ Hence it is enough to show this for $n=3$. Note also that the conclusion of the theorem is false for $n=2$. Using methods of analytic number theory the aforementioned conjecture was verified for $n\geq21$ and also for diagonal forms in five variables, see~\cite{D1, D2}.

The Oppenheim conjecture in its full generality was finally settled  affirmatively by G.~A.~Margulis in~\cite{Mar1}. Margulis actually proved a reformulation of this conjecture, in terms of closure of orbits of certain subgroup of $\SL_3(\bbr)$ on the space of unimodular lattices. This reformulation (as is well-known) is due to M.~S.~Raghunathan and indeed is a special case of Raghunathan's conjecture on the closure of orbits of unipotent groups. S.~G.~Dani and G.~A.~Margulis utilized and generalized ideas in Margulis's proof of Oppenheim conjecture to get partial result in the direction of Raghunathan's conjecture, see~\cite{DM1, DM2}. Raghunathan's orbit closure conjecture was proved by M.~Ratner in a series of path breaking papers. M.~Ratner proved a conjecture of S.~G.~Dani regarding classification of invariant measures and survived the orbit closure conjecture, see~\cite{R1, R2, R3, R4}.  

Both measure rigidity conjecture and the Oppenheim conjecture can be formulated over other local fields as well. Indeed A.~Borel and G.~Prasad in~\cite{BP}, following the same strategy as in Margulis's proof of Oppenheim conjecture, proved the analogous statement in the $S$-arithmetic setting. The measure rigidity and orbit closure in the case of product of real and $p$-adic algebaric groups were also proved by G.~A.~Margulis and G.~Tomanov~\cite{MT} and independently by M.~Ratner~\cite{R5}. The situation over a local field of positive characteristic however is far from being settled. In this paper we prove the Oppenheim conjecture over a local field of positive characteristic. Indeed our dynamical statement, Theorem~\ref{orbitclosure}, is a very special case of the orbit closure conjecture.  

Let $K$ be a global function field of characteristic $\mathfrak{p}\neq2.$ Let $\nu$ be a place of $K$ and let $\o_\nu$ be the ring of $\nu$-integers in $K$. Let $K_\nu$ be the completion of $K$ with respect to $\nu.$ Then $K_\nu$ is isomorphic to the field of power series in one variable, which will be denoted by $\theta^{-1},$ whose coefficients are in a finite field i.e. $K_{\nu}=\bbf_{\mathfrak{q}}((\theta^{-1})),$ where $\mathfrak{q}$ is a power of $\mathfrak{p}.$ Throughout the paper we let $\mathfrak{o}=\bbf_{\mathfrak{q}}[[\theta^{-1}]]$ be the valuation ring of $K_{\nu}.$ We have

\begin{thm}\label{oppenheim}
Let $Q$ be a non-degenerate isotropic quadratic form over $K_{\nu}$ in $n\geq3$ variables and assume that $Q$ is not a scalar multiple of a form whose coefficients are in $K$. Then for any $\vare>0$ there exists ${\bf v}\in\mathcal{O}_\nu^n-\{0\}$ such that $|Q({\bf v})|<\vare.$   
\end{thm}

As in the real case the if Theorem~\ref{oppenheim} holds for $n_0$ then it holds for $n\geq n_0,$ so it is enough to show this for $n=3,$ see~\cite[Proposition 1.3]{BP} for a more coherent treatment of this assertion. Note also that the conclusion of the theorem is false for $n=2$. Our proof follows closely Margulis's original proof of the Oppenheim conjecture. There are several technical difficulties which occur in carrying out the proof from characteristic zero to our setting which are dealt with in the sequel. The main  difficulty however  occurs when $\mbox{Char}(K)=3.$ It was communicated to us by K.~Mallahi~Karai~\cite{K} that he has a proof in the case ${\mbox{Char}\hh K}>3$.

In section~\ref{secorbit} we state a theorem about the closure of orbits of certain groups which is equivalent to Theorem~\ref{oppenheim}. The rest of the paper then is devoted to  the proof of Theorem~\ref{orbitclosure}. We will recall some general properties from topological dynamics in section~\ref{secminimal}. In section~\ref{secstatement} we will use the polynomial like behavior of the action of unipotent groups on the space of lattices to construct a ``polynomial like" map. This construction is essential to our proof. Indeed  this kind of constructions was used in Margulis's proof of Oppenheim conjecture and in Ratner's proof of measure rigidity conjecture, see also~\cite{Rat1}. The proof of Theorem~\ref{orbitclosure} then will be completed in section~\ref{secmain}.

\textbf{Acknowledgements.} We would like to thank Professor G.~A.~Margulis for leading us into this direction of research and for many enlightening conversations in the course of our graduate studies and also after graduation to date.      


\section{Theorem~\ref{oppenheim} and flows on homogeneous spaces}~\label{secorbit}
We observed earlier that we need to prove the theorem in the case $n=3.$ Let $G=\SL_3(K_{\nu})$ and $\Gamma=\SL_3(\mathcal{O}_\nu).$ We let $\Omega$ be the space of free $\o_\nu$-modules of determinant one in $K_{\nu}^3.$ The space $G/\Gamma$ can be naturally identified with $\Omega$ in the usual way. For any $y=g\Gamma\in\Omega=G/\Gamma$ we let $G_y=\{h\in G|\h hy=y\}=g\Gamma g^{-1}$ be the stabilizer of $y$ in $G.$  For any quadratic form $Q$ on $K_{\nu}^3,$ let $H_Q$ be the subgroup of $G$ consisting of elements which preserve the form $Q.$ We have
 
\begin{thm}\label{orbitclosure}
If $x\in G/\Gamma$ such that $\overline{H_Qx}$ is compact then $H_Qx=\overline{H_Qx}.$ 
\end{thm}
The reduction of Theorem~\ref{oppenheim} to Theorem~\ref{orbitclosure} is well-known. This was first observed by M.~S.~Raghunathan. We will reproduce the argument in our case. 

First recall that for any closed subgroup $P\subset G$ and any $y\in G/\Gamma,$ if $Py$ is closed, then the quotient space $P/{P\cap G_y}$ and the orbit $Py$ are homeomorphic. Consequently we have $P/{P\cap G_y}$ is compact if and only if $Py$ is compact.

Let $Q$ be an indefinite non-degenerate quadratic form in 3 variables as in the statement of Theorem~\ref{oppenheim}. Assume that the assertion of Theorem~\ref{oppenheim} does not hold. Thus there exits $\vare>0$ such that $|Q({\bf v})|>\vare$ for any ${\bf v}\in\o_\nu^3-\{0\}.$ Mahler's compactness criterion now implies that $H_Q\o_\nu^3$ is relatively compact in $\Omega$ i.e. $H_{Q}\Gamma$ is relatively compact in $G/\Gamma$. By Theorem~\ref{orbitclosure} above we have $H_Q\Gamma/\Gamma$ is compact hence $H_Q/{H_Q\cap\Gamma}$ is compact. Consequently $H_Q\cap\Gamma$ is a lattice in $H_Q.$ This, in view of Borel's density theorem, implies that  $H_Q\cap\Gamma$ is Zariski dense in $H_Q.$ Thus $H_Q$ is defined over $K,$ see~\cite[AG 12, 14]{B}. Note that $[Q]=\{\lambda Q:\h\lambda\in K_\nu-\{0\}\}$ is the fix set of $H_Q$ for the natural action of $G$ on the space of non-degenerate isotropic quadratic forms in $3$ variables. Since $H_Q$ is defined over $K$ we get that $[Q]$ is the set of zeros of polynomials with coefficients in $K.$ This implies that $Q$ is a (nonzero) scalar multiple of a form with coefficients in a purely inseparable extension of $K,$ see~\cite[AG 12]{B}. However $Q$ has coefficients in $K_{\nu}$ thus $Q$ is scalar multiple of a form with coefficients in $K$ which is a contradiction. 


\section{Minimal sets}\label{secminimal}
Let $G$ be an arbitrary second countable locally compact group and let $\Omega$ be a homogeneous space for $G.$ For any closed subgroup $F\subset G$ let $N_G(F)$ denote the normalizer of $F$ in $G.$

\begin{definition}\label{minimal}
Let $F$ be a closed subgroup of $G$ and $Y$ be a closed $F$-invariant subset of $\Omega$. The subset $Y$ is called $F$-minimal if it does not contain any proper closed $F$-invariant subset i.e. $Fy$ is dense in $Y$ for any $y\in Y.$
\end{definition}
It is a consequence of Zorn's lemma that any compact $F$-invariant subset of $\Omega$ contains a compact $F$-minimal subset. Let $F$ be a closed subgroup of $G$ and $Y$ a closed $F$-minimal subset of $\Omega.$ Note that if $g\in N_G(F)$ such that $gY\cap Y\neq\emptyset$ then $gY=Y.$

\begin{lem}\label{minimal1}
Let $F\subset P$ and $F\subset P'$ be closed subgroups of $G$ and let $Y$ and $Y'$ be closed subsets of $\Omega$ and let $M\subset G$. Suppose that 
\begin{itemize}
\item[(a)] $PY=Y$ and $P'Y'=Y',$
\item[(b)] $mY\cap Y'\neq\emptyset$ for any $m\in M,$
\item[(c)] $Y$ is a compact $F$-minimal subset.
\end{itemize}
Then $hY\subset Y'$ for any $h\in N_G(F)\cap\overline{P'MP}.$
\end{lem}

\begin{proof}
Define $S=\{g\in G:\h gY\cap Y'\neq\emptyset\}.$ Since $Y$ and $Y'$ are compact $S$ is a closed subset of $G.$ Note also that $P'MP\subset S$. Thus we have $\overline{P'MP}\subset S.$ Let $h\in N_G(F)\cap\overline{P'MP}$ then we have $hy=y'$ for some $y\in Y$ and $y'\in Y'.$ Since $Fh=hF$ we have $Fhy=hFy\subset Y'.$ We now take the closure and use the fact that $Y$ is $F$-minimal and get $hY\subset Y'.$
\end{proof}

\begin{cor}~\label{minimal2}
If $P=P'$ and $Y=Y'$ in the Lemma~\ref{minimal1}, then for every $h\in N_G(F)\cap\overline{PMP}$ we have $hY=Y.$ 
\end{cor}
The following is a standard fact from topological dynamics about minimal sets. We recall the proof for the sake of completeness.

\begin{lem}\label{nonclosedorbit}
Let $F$ be a closed subgroup of $G$ and let $y\in\Omega$ be such that $\overline{Fy}$ is a compact $F$-minimal subset of $\Omega$ but $F/{F\cap G_y}$ is not compact (i.e. $Fy$ is not compact). Then the closure of the subset $\{g\in G-F:\h gy\in Fy\}$ contains the identity. 
\end{lem}

\begin{proof}
We argue by contradiction. Assume the conclusion of the lemma does not hold. So we can find a relatively compact neighborhood of the identity, $\u$ say, such that $Fy\cap\hspace{.5mm}\u y=(F\cap\hh\u)y$. Let us represent $F=\cup_{n\geq1} F_n$ where $F_n\subset F_{n+1}$ are compact subsets of $F$. Recall that $Fy$ is not compact thus for any $n\geq1$ we can find $z_n\in Fy$ such that $F_nz_n\cap (F\cap\hh\u)y=\emptyset.$ Since $F_nz_n\subset Fy$ we have $F_nz_n\cap\hh\u y=\emptyset$. Let $\Psi=\cup_n F_nz_n.$ The aforementioned properties imply that the closure of $\Psi$ does not contain $y.$ However if we replace $\{z_n\}$ by a subsequence, if necessary, we may and will assume $z_n\rightarrow z\in Y.$ Note now that $Y$ is $F$-minimal and that $Fz\subset\overline{\Psi}$ thus $y\in\overline{\Psi},$ which is a contradiction. 
\end{proof}


\section{Some preliminary statements}\label{secstatement}
For any $\mu\in K_{\nu}-\{0\}$ let us denote by $H_{\mu}$ the subgroup of $G$ consisting of elements which preserve the form $Q_{\mu}(x)=2x_1x_3-\mu x_2^2.$ Note that any isotropic non-degenerate quadratic form $Q,$ in 3 variables has a two dimensional hyperbolic  
space i.e. $Q$ is congruent to $Q_{\mu}$ for some $\mu.$ Thus there exists $g_Q\in G$ and $\mu_{Q}\in K_{\nu}-\{0\}$ such that $H_{\mu_Q}=g_QH_Qg_Q^{-1}.$ Hence it suffices to prove Theorem~\ref{orbitclosure} for $H_\mu.$
We denote $H=H_1.$ We prove Theorem~\ref{orbitclosure} in the case $\mu=1,$ the proof for arbitrary $\mu$ is identical.  Let us fix some notations to be used throughout the paper. As above let $H=H_1$ and define 
$$d(t)=\left(\begin{array}{ccc}t & 0 & 0\\ 0 & 1 & 0\\ 0 & 0 & t^{-1}\end{array}\right),\h\h\h v_1(t)=\left(\begin{array}{ccc}1 & t & t^2/2\\ 0 & 1 & t\\ 0 & 0 & 1\end{array}\right)$$
$$v_2(t)=\left(\begin{array}{ccc}1& 0 & t\\ 0 & 1 & 0\\ 0 & 0 & 1\end{array}\right)$$
Let $D(t)=\{d(t):\h t\in K_{\nu}^{\times}\}$, $V_1=\{v_1(t):\h t\in K_{\nu}\}$ and $V_2=\{v_2(t):\h t\in K_{\nu}\},$ more generally if for any subset $A\subset K_\nu$ we let $V_2(A)=\{v_2(t):\h t\in A\}.$ 
If $f:K_{\nu}\rightarrow K_{\nu}$ is a polynomial map we let $V_2(f)=V_2(\mbox{Im}(f)).$  Let
$$V=V_1\cdot V_2=\left\{\left(\begin{array}{ccc}1 & a & b\\ 0 & 1 & a\\ 0 & 0 & 1\end{array}\right):\h a, b\in K_{\nu}\right\}$$
Note that $DV$ is the normalizer of $V_1$ in $G$ and that $DV_1$ is the intersection of $H$ with the group of upper triangular matrices. Let $W$ (resp. $W^-$) be the set of strictly upper triangular (resp. lower triangular) matrices in $G,$ and also let $B$ (resp. $B^-$) be the set of upper triangular (resp. lower triangular) matrices in $G.$ For any subgroup $A$ of $W$ which is normalized by $d(\theta^3)$ we let $$A_k=d(\theta^3)^k(W(\mathfrak{o})\cap A)d(\theta^3)^{-k}$$ 
The full diagonal group will be denoted by $T.$ We let $L=B^-W_1$ where 
$$W_1=\left\{\left(\begin{array}{ccc}1& 0 & a\\ 0 & 1 & b\\ 0 & 0 & 1\end{array}\right):\h a,b\in K_{\nu}\right\}$$
this is a rational cross-section for $V_1$ in $G.$ 
Let $V_1^-,\h W_1^-,\h V_2^-$ denote the transpose of $V_1,\h W_1,\h V_2$ respectively.  

\vspace{1mm}
Let $x\in\Omega$ be such that $\overline{Hx}$ is compact and $Hx\neq\overline{Hx}.$ Fix once and for all a compact $H$-minimal subset $X$ of $\overline{Hx}$ and let $Y\subset X$ be a compact $V_1$-minimal subset.

\begin{lem}\label{Dy-compact}
Let $y\in\Omega$ be such that $Dy$ is relatively compact in $\Omega.$ Then $W\cap G_y=\{e\},$ therefore $U/{U\cap G_y}$ is not compact for any non-compact subgroup $U$ of $W.$
\end{lem}

\begin{proof}
This is a consequence of the following two facts.

(i) $W=\{g\in G:\h d(t)g\hspace{.5mm}d(t)^{-1}\rightarrow e\h\mbox{as}\h t\rightarrow0\}$

(ii) If $\gamma\in G_y,$ $\gamma\neq e$ and $\{d_n\}$ is a sequence in $G$ such that $d_n\gamma d_n^{-1}\rightarrow e$ then the set $\{d_ny\}$ is not relatively compact in $\Omega.$
\end{proof}

Lemma~\ref{Dy-compact} has he following consequence which is a supped up version of Lemma~\ref{nonclosedorbit} in the form which needed for our construction in Lemma~\ref{gettingextrainv}. 

\begin{lem}\label{displacement}
The closure of $\{g\in G-DV:\h gY\cap Y\neq\emptyset\}$ contains $e.$
\end{lem}

\begin{proof}
Note first that since $DV=N_G(V_1)$ and $Y$ is $V_1$-minimal we have $gY=Y$ if $g\in DV$ and $gY\cap Y\neq\emptyset$. Thus we have 
$$\S=\{g\in DV:\h gY\cap Y\neq\emptyset\}=\{g\in DV:\h gY=Y\}$$
Which says $\S$ is a closed subgroup of $G$ and indeed $V_1\subset\S$. Assume now that the contrary to the lemma holds. Since $Y$ is $\S$-minimal and compact, in view of Lemma~\ref{nonclosedorbit} we get $\S/{\S\cap G_y}$ is compact. In particular $\Lambda=\S\cap G_y$ is a lattice in $\S.$ Since $V_1\subset\S$ and $V_1$ is normal in $DV=(DV_2)V_1$ we may write $\S=(\S\cap DV_2)V_1.$ Let $\S_2=\S\cap DV_2.$ We claim that $\pi(\Lambda)$ is discrete where $\pi:\S\rightarrow\S/V_1$ is the natural projection. Let us assume the claim for a second. Then we get $\Lambda V_1$ is a closed subgroup and since $\Lambda$ is a co-compact lattice in $\S$ this implies that $\Lambda\cap V_1$ is a lattice in $V_1.$ Note however that $Dy\subset X$ is relatively compact hence in view of the Lemma~\ref{Dy-compact} this is a contradiction and the lemma will be concluded. We now show the claim. First note that $(DV)(\mathfrak{o})$ has a neighborhood of identity which is a pro-$\mathfrak{p}$ group. We call this neighborhood $(DV)_{\mathfrak{p}}$ and let $\S_{2\mathfrak{p}}=\S_2\cap (DV)_{\mathfrak{p}}.$ We will show that $\pi(\Lambda)\cap\pi(\S_{2\mathfrak{p}})=\{e\}.$ Assume the contrary then there exists some  $\lambda\in\Lambda$ such that $e\neq\pi(\lambda)\in\pi(\S_{2\mathfrak{p}}).$ Note that $V_1$ has a filtration, say $\{V_{1n}\},$ by pro-$\mathfrak{p}$ groups which are all normalized by $\S_{2\mathfrak{p}}.$ Hence $\lambda\in\S_{2\mathfrak{p}}V_{1n}$ for some $n.$ This is a pro-$\mathfrak{p}$ group. The group generated by $\langle\lambda\rangle\subset\Lambda$ is a discrete subgroup of this pro-$\mathfrak{p}$-group which implies that it is a finite $\mathfrak{p}$-group. Write $\lambda=dv,$ where $d\in D$ and $v\in V.$ Note that in view of the Lemma~\ref{Dy-compact} we have $\lambda\not\in V$ thus $d=d(t)\neq e.$. Modding out by the normal subgroup $V$ of $DV$ this implies that $d$ has order a power of $\mathfrak{p}$ which implies $d=e$. This contradiction establishes the claim.
\end{proof}

We will need a construction from~\cite{MT}, see also~\cite{Mar1} and~\cite{R1, Rat1}. As was mentioned one of the main ingredients in the proof of Theorem~\ref{orbitclosure} is the polynomial like behavior of the way two unipotent orbits diverge from each other. Since one needs to deal with quasi-affine spaces rather than affine spaces we do not quite get polynomials but rather certain rational maps which become polynomial if we embed our quasi-affine space into an affine space. The following definition, from~\cite{MT}, is the precise formulation.

\begin{definition}\label{quasiregular}(cf.~\cite[Definition 5.3]{MT})
\begin{itemize}
\item[(i)]Let $\mathbb{E}$ be a $K_{\nu}$-algebraic group, $F$ a $K_{\nu}$-algebraic subgroup of $\mathbb{E}(K_{\nu})$ and $\mathbb{M}$ a $K_{\nu}$-algebraic variety. A $K_{\nu}$-rational map $f:\mathbb{M}(K_{\nu})\rightarrow\mathbb{E}(K_{\nu})$ is called $F$-{\it quasiregular} if the map from $\mathbb{M}(K_{\nu})$ to $\mathbb{V}$ given by $x\mapsto\rho(f(x)p)$ is $K_{\nu}$-regular for every $K_{\nu}$-rational representation $\rho:\mathbb{E}\rightarrow\mbox{GL}(\mathbb{V})$ and every point $p\in\mathbb{V}(K_{\nu})$ such that $\rho(F)p=p.$
\item[(ii)] Let $\mathbb{W}\subset\mathbb{E}$ be a $K_\nu$-split unipotent subgroup, see~\cite[Chapter 5]{B} for the definition and properties of split unipotent groups. (It is worth mentioning that what we need here is that as a $K_\nu$-variety, $\mathbb{W}$ is $K_\nu$-isomorphic to an affine space.) Let $E=\mathbb{E}(K_{\nu})$ and $\mathcal{W}=\mathbb{W}(K_\nu).$ A map $\phi:\mathcal{W}\rightarrow E$ is called {\it strongly} $\mathcal{W}$-{\it quasiregular} if there exist
\begin{itemize}
\item[(a)] a sequence $g_n\in E$ such that $g_n\rightarrow e.$
\item[(b)]  a sequence $\{\alpha_n:\mathcal{W}\rightarrow\mathcal{W}\}$ of $K_{\nu}$-regular maps.
\item[(c)]  a sequence $\{\beta_n:\mathcal{W}\rightarrow\mathcal{W}\}$ of $K_{\nu}$-rational maps.
\item[(d)] a Zariski open nonempty subset $\mathcal{X}\subset\mathcal{W}$
\end{itemize}
such that $\phi(u)=\lim_{n\rightarrow\infty}\alpha_n(u)g_n\beta_n(u)$ and the convergence is uniform on the compact subsets of $\mathcal{X}$
\end{itemize}
\end{definition}

\noindent
Note that if $\phi$ is strongly $\mathcal{W}$-quasiregular then it indeed is $\mathcal{W}$-quasiregular. Let $\rho:E\rightarrow\mbox{GL}(\Phi)$ be a $K_{\nu}$-rational representation and let $p\in\Phi$ be a $\mathcal{W}$-fixed vector. For any $u\in\mathcal{X}$ we have $$\rho(\phi(u))p=\lim_{n\rightarrow\infty}\rho(\alpha_n(u)g_n)p.$$ 
Identify $\mathcal{W}$ with an affine space, as we may, thanks to the fact $\mathcal{W}$ is split. The sequence $\{\psi_n:\mathcal{W}\rightarrow\Phi,\hspace{2mm}u\mapsto\rho(\alpha_n(u)g_n)p\}$ is a sequence of polynomial maps of bounded degree and also the family is uniformly bounded on compact sets so it converges to a polynomial map with coefficients in $K_{\nu}$. This says $\phi$ is $\mathcal{W}$-quasiregular.

The following is an important application of the polynomial like behavior of the action of $V_1$ on $\Omega.$ Actually later on we will need it for some other subgroup which share similar features with $V_1$ i.e. split unipotent algebraic subgroups of $G$ after change of the base field. The proof in the more general setting is the same as it is clear from the proof given here.

\begin{lem}\label{gettingextrainv}
Let $\{g_n\}\subset G-DV$ be such that $g_n\rightarrow e.$ Then $N_G(V_1)\cap\overline{V_1\{g_n\}V_1}$ contains the image of a non-constant strongly $V_1$-quasiregular map $\phi$. Furthermore $\imp\not\subset KV_1$ for any compact subset $K\subset G.$
\end{lem}

\begin{proof}
Let $\{g_n\}\subset G-DV$ be such that $g_n\rightarrow e.$\vspace{1mm} We define the rational morphisms $\tilde{\phi}_n:V_1\rightarrow L$ and $\omega_n:V_1\rightarrow V_1$ to be the maps so that $v_1(t)g_n=\tilde{\phi}_n(t)\omega_n(t)$ holds for all $v_1(t)$ in a Zariski open dense subset of $V_1.$ 

By Chevalley's Theorem there exits a $K_{\nu}$-rational representation $\rho:G\rightarrow\mbox{GL}(\Phi)$ and $q\in\Phi$ such that
$$V_1=\{g\in G:\h\rho(g)q=q\}\h\h\mbox{and}\h\h N_G(V_1)=\{g\in G:\h \rho(V_1)\rho(g)q=\rho(g)q\}$$
Let $\mathcal{B}(v)\subset\rho(G)q$ be a bounded neighborhood of $q$ in V. Since on $\{g_n\}\not\in N_G(V_1)$ we have; there exists a sequence of positive integers $\{r(n)\}$ with $r(n)\rightarrow\infty$ such that $V_{1r(n)}g_nq\not\subset\mathcal{B}(q)$ and $V_{1k}g_nq\subset\mathcal{B}(q)$ for all $k<r(n),$ where
$$V_{1n}=\{v_1(t):\h t\in K_{\nu}\h\mbox{and}\h|t|<\mathfrak{q}^{3n}\}$$
For any $n\in\bbn$ let $\alpha_n:V_1\rightarrow V_1$ be the conjugation by $d(\theta^3)^{r(n)}$. Define the $K_{\nu}$-rational maps $\phi_n$ by $\phi_n=\tilde{\phi}_n\circ\alpha_n:V_1\rightarrow L.$ Let $$\phi'_n=\rho_L\circ\phi_n:V_1\rightarrow \Phi$$

We have $\phi'_n(t)=\alpha_n(v_1(t))g_nq,$ thus $\phi'_n:V_1\rightarrow \Phi$ is a $K_{\nu}$-regular map whose degree is independent of $n$. This is to say $\{\phi'_n\}$ is a set of polynomial maps of bounded degree. Using the definition of $\phi'_n$ we also have $\{\phi'_n\}$ is a uniformly bounded family of polynomials. Thus passing to a subsequence, which we will still denote it by $\phi'_n,$ there is a polynomial map $\phi':V_1\rightarrow \Phi$ such that 
$$\phi'(t)=\lim_{n\rightarrow\infty}\phi'_n(t)\hspace{3mm}\mbox{for every}\hspace{3mm}t\in K_{\nu}$$  
Note that $\phi'(e)=q$ as $g_n\rightarrow e$ and that $\phi'$ is non-constant since $g_n\not\in N_G(V_1)$.

Recall that $L$ is a rational cross-section for $G/V_1$ which contains $e.$ Thus $L$ gets mapped onto a Zariski open dense subset $\mathcal{M}$ of Zariski closure of $\rho(G)q$ and that $q\in\mathcal{M}.$ Hence we can define a $K_{\nu}$-rational map $\phi:V_1\rightarrow L$ by 
$$\phi=\rho_L^{-1}\circ\phi'$$
The construction above gives $\phi(e)=e$ and $\phi$ is non-constant.

\vspace{1mm}
We now show; the map $\phi$ satisfies the conditions of the lemma i.e. 
\begin{itemize}
\item[(i)] $\phi$ is strongly $V_1$-quasiregular 
\item[(ii)] $\mbox{Im}(\phi)\subset N_G(V_1)$.
\item[(iii)] $\imp\not\subset KV_1$ for any compact subset $K\subset G.$
\end{itemize}

Note that by above construction we have if $v_1(t)\in\phi'^{-1}(\mathcal{M})$ then
$$\phi(t)=\lim_{n\rightarrow\infty}\phi_n(t)$$
and the convergence above is uniform on the compact set of $\phi'^{-1}(\mathcal{M}).$ We have 
$$\phi_n(t)=\alpha_n(v_1(t))g_n\beta_n(t)\hspace{3mm}\mbox{where}\hspace{3mm}\beta_n(t)=\omega_n(\alpha_n(v_1(t)))^{-1}$$
Above says for $v_1(t)\in\phi'^{-1}(\mathcal{M})$ we can write
$$\phi(t)=\lim_{n\rightarrow\infty}\alpha_n(v_1(t))g_n\beta_n(t),$$ 
this establishes (i).

To prove (ii) above recall that $N_G(V_1)=\{g\in G:\h \rho(V_1)\rho(g)q=\rho(g)q\}$. We remarked above that $\phi(t)=\lim_{n\rightarrow\infty}\alpha_n(v_1(t))g_n\beta_n(t).$ Let $v_1(s)\in V_1$\vspace{1mm} be an arbitrary element we need to show $\rho(v_1(s))\rho(\phi(t))q=\rho(\phi(t))q$. Note that
$$\rho(v_1(s)\alpha_n(v_1(t))g_n)q=\rho(\alpha_n(\alpha_{n}^{-1}(v_1(s))t)g_n)q$$
The result is immediate now if we note that $\alpha_n^{-1}(v_1(s))\rightarrow e$ as $n\rightarrow\infty.$ This finishes the proof of (ii).

To see (iii) note that $\phi=\rho_L^{-1}\circ\phi'$ and $\phi'$ is a non-constant (hence unbounded) polynomial map.
\end{proof}


For later use we need to explicitly determine various polynomials which where constructed in Lemma~\ref{gettingextrainv}. Let
\begin{equation}\label{e;gn}g_n=\left(\begin{array}{ccc}a_{1n}& a_{2n} & a_{3n}\\ b_{1n} & b_{2n} & b_{3n}\\ c_{1n} & c_{2n} & c_{3n}\end{array}\right)\end{equation}
Then we have 
$$v_1(t)g_n=\left(\begin{array}{ccc}a_{1n}+b_{1n}t+c_{1n}\frac{t^2}{2}& a_{2n}+b_{2n}t+c_{2n}\frac{t^2}{2} & a_{3n}+b_{3n}t+c_{3n}\frac{t^2}{2}\\ b_{1n}+c_{1n}t & b_{2n}+c_{2n}t & b_{3n}+c_{3n}t\\ c_{1n} & c_{2n} & c_{3n}\end{array}\right)$$ 
Write $\alpha_n(v_1(t))g_n\beta_n(t)=\sigma_n(t)\vt_n(t)$ where $\sigma_n:V_1\rightarrow B^-$ and $\vt_n:V_1\rightarrow W_1$ are $K_{\nu}$-rational morphisms. That is
\begin{equation}\label{e;sigma-vt}\sigma_n(t)=\left(\begin{array}{ccc}\sigma_{11}^n(t)& 0 & 0\\ \sigma_{21}^n(t) & \sigma_{22}^n(t) & 0\\ \sigma_{31}^n(t) & \sigma_{32}^n(t) & \sigma_{33}^n(t)\end{array}\right)\h\h\mbox{and}\h\h\vt_n(t)=\left(\begin{array}{ccc}1& 0 &\vt_{13}^n(t)\\ 0 & 1 & \vt_{23}^n(t)\\ 0 & 0 & 1\end{array}\right)\end{equation}
Note that our construction above says $\phi(t)\in L\cap N_G(V_1)=DV_2$ so all we need from above are the maps $\{\sigma_{11}^n(t)\}$ and $\{\vt_{13}^n(t)\}$ which are easily calculated. Let $\alpha_n(t)=\theta^{3r(n)}t.$ Then
\begin{itemize} 
\item[($\dagger$)]$\sigma_{11}^n(t)=a_{1n}+b_{1n}\alpha_n(t)+c_{1n}\frac{\alpha_n(t)^2}{2}=a_{1n}+b'_{1n}t+c_{1n}''\frac{t^2}{2}$
\item[($\ddagger$)]$\vt_{13}^n(t)=\frac{2(v_1(\alpha_n(t))g_n)_{13}(v_1(\alpha_n(t))g_n)_{11}-((v_1(\alpha_n(t))g_n)_{12})^2}{2((v_1(\alpha_n(t))g_n)_{11})^2}=\frac{\vt_{0n}(t)}{(\sigma_{11}^n(t))^2},$ where $\vt_{0n}$ is a degree 4 polynomial.
\end{itemize}
Passing to the limit we have 
\begin{equation}\label{e;phi-expl}\phi(t)=d(\sigma(t))\h v_2(\vt_0(t))=d(\sigma(t))\h v_2\left(\frac{\vt(t)}{\sigma(t)^2}\right)\end{equation}
where $\sigma(t)$ is a polynomial of degree at most 2 and $\vt(t)$ is a polynomial of degree at most 4. Moreover $d(\sigma(0))=v_2(\vt_0(0))=e.$

In the sequel we need some more properties of the map constructed in Lemma~\ref{gettingextrainv}. Let us fix some notation. 
It follows from standard facts in algebraic group theory, see~\cite{B}, that the product map defines an isomorphism between $W^-\times T\times W$ and a Zariski open dense subset of $G$ which contains $e.$ In particular there exists an open neighborhood of the identity in $G$ such that for all $g$ in that neighborhood
$$g=W^-(g)T(g)W(g)=V_1^-(g)W_1^-(g)T(g)W_1(g)V_1(g)$$
where $V_1^-(g)\in V_1^-,\h W_1^-(g)\in W_1^-,\h T(g)\in T,\h V_1(g)\in V_1,\h W_1(g)\in W_1,\h W^-(g)=V_1^-(g)W_1^-(g)$ and $W(g)=W_1(g)V_1(g).$ The following lemma follows from a more general result proved in~\cite[Proposition 6.7]{MT} in characteristic zero case. That proof works in our setting also however our particular case here allows a hands on proof. We give this proof here for the sake of completeness and refer to~\cite[Proposition 6.7]{MT} for a conceptual proof.      

\begin{lem}\label{l;image-w}
Let $\{g_n\} \subset W_1^-B-DV$ be a sequence such that $g_n\rightarrow e.$ Let $\phi$ be the $V_1$-quasiregular map constructed in Lemma~\ref{gettingextrainv} using $\{g_n\}$. Then ${\rm Im}(\phi)\subset W.$
\end{lem}

\begin{proof}

The proof of Lemma~\ref{gettingextrainv} implies that $\phi(t)=\lim_n\alpha_n(v_1(t))g_n\beta_n(t)$ for all $t$ such that $v_1(t)\in(\phi')^{-1}(\mathcal{M}).$ The calculation following Lemma~\ref{gettingextrainv} gives $\phi(t)=\lim_n\sigma_n(t)\vt_n(t)$ for all $v_1(t)\in(\phi')^{-1}(\mathcal{M})$ where $\sigma_n$ and $\vt_n$ are as in~\eqref{e;sigma-vt}. As was mentioned above we need only $\sigma^n_{11}$ and $\vt^n_{13}.$ Our assumption, $g_n\in W^-_1B,$ is to say that  in~\eqref{e;gn} $b_{1n}=0$ for all $n.$ This using the above calculations implies that $\sigma^n_{11}(t)=a_{1n}+c_{1n}t^2/2$ and 
\begin{equation}\label{e;vt-explicit}\vt_{13}^n(t)=\frac{2(a_{1n}+c_{1n}\frac{t^2}{2})(a_{3n}+b_{3n}t+c_{3n}\frac{t^2}{2})-(a_{2n}+b_{2n}t+c_{2n}\frac{t^2}{2})^2}{2(a_{1n}+c_{1n}\frac{t^2}{2})^2}\end{equation}
If we expand~\eqref{e;vt-explicit}, we get: in the expression for the nominator the coefficient of $t^3$ is $b_{3n}c_{1n}-b_{2n}c_{2n}$ and the coefficient of $t^4$ is $(2c_{1n}c_{3n}-c_{2n}^2)/4.$ Now $\sigma(t)$ is non-constant if and only if $\lim_n\theta^{3r(n)}c_{1n}\neq0.$ However this in view of the fact that $g_n\rightarrow e$ implies that either $\lim_n\theta^{3r(n)}(b_{3n}c_{1n}-b_{2n}c_{2n})\neq0$ or$\lim_n\theta^{3r(n)}(2c_{1n}c_{3n}-c_{2n}^2)\neq0.$ Either case implies that  $\{\vt_{13}^n(\alpha_n(\alpha_n(t))\}$ diverges. This is a contradiction thus $\phi(t)\subset W$ for all $t$ such that $v_1(t)\in(\phi')^{-1}(\mathcal{M}).$ Since this is a Zariski dense subset of $V_1$ the lemma follows.  
\end{proof}

Let the notation and conventions be as before. Recall in particular that $x\in G/\Gamma$ such that $\overline{Hx}$ is compact and $\overline{Hx}\neq Hx.$ We fixed $X\subset\overline{Hx}$ an $H$-minimal subset and $Y\subset X$ a $V_1$-minimal subset.

\begin{prop}\label{cases}
At least one of the following holds
\begin{itemize}
\item[(i)] There exists a non-constant polynomial, $\sigma(t),$ and a polynomial of degree at most four, $\vt(t),$ such that $\phi(t)Y=Y$ where $\phi(t)=d(\sigma(t))\hh v_2\left(\frac{\vt(t)}{\sigma(t)^2}\right).$  
\item[(ii)] $V_2\hh y\subset\overline{Hx}$ for some $y\in\overline{Hx}.$ 
\item[(iii)]Characteristic of $K$ equals 3 and $Y$ is $V_2(f)$-invariant, where $f(t)=at^3$ for some $a\in K_{\nu}-\{0\}.$
\end{itemize}   
\end{prop}

\begin{proof}
Using Lemma~\ref{displacement} above we can find $\{g_n\}\subset G-DV$ such that $g_n\rightarrow e$ and that $g_nY\cap Y\neq\emptyset.$ Applying Lemma~\ref{minimal2} with $P=F=V_1,$ $Y=Y$ and $M=\{g_n\}$ one has $hY=Y$ for every $h\in N_G(V_1)\cap\overline{V_1MV_1}.$ Note that we are in the situation of Lemma~\ref{gettingextrainv}, using that lemma and the calculation after loc. cit. in particular~\eqref{e;phi-expl} we have; there are polynomials $\sigma(t)$ of degree at most 2 and $\vt(t)$ of degree at most 4 such that $\phi(t)Y=Y,$ where $\phi(t)=d(\sigma(t))\h v_2\left(\frac{\vt(t)}{\sigma(t)^2}\right).$ We may assume $\sigma(t)$ is a constant polynomial else (i) above holds and there is nothing to show.

Thus we may and will assume that $\sigma(t)=\sigma(0)=1$ for all $t\in K_{\nu}$ i.e.  $\phi(t)=v_2(\vt(t)).$ There are 3 possibilities for $\vt'(t),$ the derivative $\vt(t);$
\begin{itemize}
\item[(a)] $\vt'(t)=0$ for all $t\in K_{\nu}.$ Note that $\vt(t)$ is a non-constant polynomial of degree at most 4 hence this can only happen if ${\mbox{Char}\hh K}=3$ (recall that $\char\hh K\neq2$). Since $\vt(0)=0$ this says $\vt(t)=at^3$ for some nonzero $a\in K_{\nu},$ which means case (iii) of the proposition holds.
\item[(b)] $\vt'(t)$ is a non-constant polynomial. 
\item[(c)] $\vt'(t)$ is constant but it is not zero. This in view of the same considerations as in (b) implies $\vt(t)=at^3+bt\hh$ where $a,b\in K_{\nu}$ and $b\neq0.$
\end{itemize} 
We will now show (ii) holds in either cases (b) and (c). The occurrence of these, thanks to the inverse function theorem, implies that the image of $\vt$ contains some open set. Since $V_2(\vt(t))Y=Y$ the same holds true for the group generated by $\{v_2(\vt(t)):\h t\in K_\nu\}.$ Thus we get that there exists some open neighborhood $\mathfrak{o}'$ of the origin such that $Y$ is invariant under $V_2(\mathfrak{o}').$ Let $z\in Y.$ Since $X$ is $H$-invariant and $V_2(\mathfrak{o}')z\in Y$ we have $d(\theta^3)^nV_2(\mathfrak{o}')z\subset X$ for any $n\in\bbn$. We have 
$$d(\theta^3)^nV_2(\mathfrak{o}')z=d(\theta^3)^nV_2(\mathfrak{o}')d(\theta^3)^{-n}d(\theta^3)^nz=V_{2n}d(\theta^3)^nz$$
Let $z_n=d(\theta^3)^nz.$ Since $X$ is compact there is $y\in X$ such that $z_{n}\rightarrow y.$ Note also that $V_{2n}\subset V_{2(n+1)}$ are compact sets and $V_2=\cup_{n\geq1}V_{2n}$. Thus $V_2\hh y\subset X$. 
\end{proof}

The following lemma describes the situation when (i) in Proposition~\ref{cases} holds. 

\begin{lem}\label{l;sigma-noncon}
If (i) in Proposition~\ref{cases} holds, then there exists $y\in \overline{Hx}$ and a polynomial $f$ whose derivative is non-constant such that $V_2(f)y\subset\overline{Hx}.$
\end{lem}
Before starting the proof it is worth mentioning that in this case we merely get $V_2(f)y\subset\overline{Hx}$ and not $V_2(f)y\subset Y$ for some $y\in Y.$ Thus we cannot replace the conclusion with the stronger assertion $V_2y\subset\overline{Hx}.$

\begin{proof}
We have $Y$ is invariant under 
$$R=\langle d(\sigma(t))\h v_2\left(\frac{\vt(t)}{\sigma(t)^2}\right):\h t\in K_{\nu}\rangle $$
where $\langle\bullet\rangle$ denotes the group generated by $\bullet.$ Recall also that  $\sigma(t)$ is a non-constant polynomial.

Recall that $Hx$ is not closed. We claim that for any $y\in Y$ the closure of the subset $M=\{g\in G-H:\h gy\in\overline{Hx}\}$ contains the identity.
Assume the contrary. In particular since $y\in\overline{Hx}$ we have; the closure of $\{g\in G-H:\h gy\in Hy\}$ does not contain the identity. Then since $y\in X$ and $X$ is $H$-minimal Lemma~\ref{nonclosedorbit} implies that $Hy$ is compact. Also note that since $y\in\overline{Hx}$ and $e\not\in\overline{M}$ we have $y\in Hx.$ These imply that $Hx$ is closed which is a contradiction. 

We now apply Lemma~\ref{minimal1} with $M,$ $Y'=\overline{Hx},\h Y,\h P'=H,\h F=V_1$ and $P=V_1.$ Hence for any $h\in N_G(V_1)\cap HMV_1$ we have $hY\subset Y'=\overline{Hx}.$ Pick $\{g_n\}\subset M$ such that $g_n\rightarrow e.$ Let us first assume that there exists some  subsequence $\{g_{n_i}\}\subset G-HV_2.$ Abusing the notation we continue to denote this subsequence by $\{g_n\}.$ For $n$ large enough we have $g_n=V_1^-(g_n)W_1^-(g_n)T(g_n)W_1(g_n)V_1(g_n).$
Multiplying on left by $H$ we may and will assume that $V_1^-(g_n)=e.$ Thus we have a sequence $\{g_n\}$ such that $g_n\rightarrow e,$ $V_1^-(g_n)=e$ and $\{g_n\}\not\subset HV_2.$ We now apply the construction of Lemma~\ref{gettingextrainv} with $\{g_n\}.$  We get a non-constant strongly $V_1$-quasiregular map $\phi_1$ such that ${\rm Im}(\phi_1)\subset N_G(V_1)$ and $\phi_1(s)Y\subset Y'$ for all $s\in K_\nu.$

Recall that $g_n\in W_1^-TW=W_1^-B$ hence Lemma~\ref{l;image-w} implies that $\mbox{Im}(\phi_1)\subset W.$ This in view of the fact $N_G(V_1)=DV$ gives $\mbox{Im}(\phi_1)\subset V_2$ i.e. $\phi_1(s)=v_2(\vt_1(s))$ and for any $s\in K_{\nu}$ we have $v_2(\vt_1(s))Y\subset X.$ Define the polynomial $f_u(t)=\vt_1(s)\sigma^2(t)-\vt(t),$ where $u=v_2(\vt_1(s)).$  Since $\vt_1(s)$ and $\sigma(t)$ are non-constant polynomials and $\mbox{deg}(\sigma(t))\leq2$, we can find $u_0=\phi_1(s_0)$ such that both $f_{u_0}$ and also its derivative $f'_{u_0}$ are non-constant polynomials. 

We now turn to the case $\{g_n\}\subset HV_2.$ Thus there are infinitely many elements $v_n=v_2(t_n)\in V_2$ such that $v_ny\in\overline{Hx}$. Thus we may find $u_0=v_2(t_{k})$ for some $k$ such that $f_{u_0}(t)=t_k\sigma^2(t)-\vt(t)$ and $f'_{u_0}$ are both non-constant polynomials. 

Now apply Lemma~\ref{minimal1} with $M=\{u_0\}$ which is defined above, $Y'=\overline{Hx},\h Y,\h P'=H,\h F=V_1$ and $P=\langle R,V_1\rangle,$ the closed group generated by $R$ and $V_1.$ According to Lemma~\ref{minimal1} we have $hY\subset\overline{Hx}$ for any $h\in N_G(V_1)\cap\overline{HMP}.$ Note that for any $t$ one has
$$d(\sigma(t))\hh u_0\left[d(\sigma(t))\hspace{.5mm} v_2\left(\frac{\vt(t)}{\sigma(t)^2}\right)\right]^{-1}=v_2(f_{u_0}(t))\in HMP$$
Thus $v_2(f_{u_0}(t))Y\subset\overline{Hx}$ for any $t\in K_{\nu}.$
\end{proof}


\section{Proof of Theorem~\ref{orbitclosure}}\label{secmain}

We will prove Theorem~\ref{orbitclosure} in this section using the construction in Section~\ref{secstatement}. We will first reduce the proof of Theorem~\ref{orbitclosure} to the ``inseparable" case. The existence of this inseparable case in some sense is the main difference with the proof in the characteristic zero case. We begin with the following elementary observation

\begin{lem}\label{l;elementary}
Let $L$ be a lattice in $\Omega$ and $B(t_0,r_0)$ an open ball of radius $r_0$ about $t_0$ in $K_\nu.$ Then there exists some nonzero $\vbf\in V_2(B(r_0,t_0))L$ such that $Q(\vbf)=0.$
\end{lem}

\begin{proof}
Indeed we may and will assume $r_0<1.$ Let $B=B(t_0,r_0).$ For any $\wbf\in K_\nu^3$ and $r>0$ let $B(\wbf,r)$ denote the ball of radius $r$ about $\wbf$ in $K_\nu^3.$ For any $\ubf=(u_1,u_2,u_3)\in K_\nu^3$ with $u_3\neq0$ define $q(\ubf)=-\frac{Q(\ubf)}{2u_3^2}.$ Fix some $\wbf_r=(-t_0w_3,0,w_3)$ with $|w_3|>\max\{r/r_0,|t_0|r/r_0\}.$ Then the map $q:B(\wbf_r,r)\rightarrow K_\nu$ is defined on $B(\wbf_r,r)$ and $\mbox{Im}(q)\subset B.$ Since $L$ is a lattice in $K_\nu^3$ there exists some $r=r(L)>0$ such that $L\cap B(\wbf,r)\neq\emptyset$ for all $\wbf\in K_\nu^3.$ Let $\wbf_r$ be given as above corresponding to this $r.$ Let $\ybf\in L\cap B(\wbf_r,r).$ Then $y_3\neq0$ and $s=q(\ybf)=-\frac{Q(\ybf)}{2y_3^2}\in B.$ 

For any $\ubf\in K_{\nu}^3$ we have 
$$v_2(t)\mathbf{u}=v_2(t)\left(\begin{array}{l} u_1\\ u_2\\ u_3\end{array}\right)=\left(\begin{array}{c} u_1+tu_3\\ u_2\\ u_3\end{array}\right)$$
Hence $Q(v_2(t)\mathbf{u})=2tu_3^2+Q(\mathbf{u}).$ Thus $Q(v_2(s)\ybf)=0.$
\end{proof}

\begin{cor}~\label{easycases}
Theorem~\ref{orbitclosure} holds if (i) or (ii) in Proposition~\ref{cases} holds. In particular Theorem~\ref{orbitclosure} holds if $\char\hh K>3.$ 
\end{cor}

\begin{proof}
By Lemma~\ref{l;elementary} the corollary is immediate if case (ii) holds and follows from  
Lemma~\ref{l;sigma-noncon} and the inverse function theorem if case (i) holds. 
\end{proof}

\textbf{The inseparable case.} The Corollary~\ref{easycases} reduces the proof of Theorem~\ref{orbitclosure} to the case where (iii) in Proposition~\ref{cases} holds. In particular from now on we assume ${\char\hh K}=3.$ We keep all the assumptions and the notation as before. Let us fix some further notation. Define $K_{\nu}^{(3)}=\{k^3|\h k\in K_{\nu}\}.$
This is a subfield of $K_{\nu}$ and $K_{\nu}/K_{\nu}^{(3)}$ is a purely inseparable extension of degree $3.$ For any nonzero element $a\in K_{\nu}$ we will let $V_2^a=V_2(f)$ where $f(t)=at^3$. The subgroups $V_2^a$ and $V^a=V_1V_2^a$ are closed subgroups of $V.$ They are actually unipotent algebraic groups if we change our base field to $K_{\nu}^{(3)}.$ To be more precise one needs to use Weil's restriction of scalars and replace $\mbox{SL}_3$ by $\mbox{R}_{{K_\nu}/K_{\nu}^{(3)}}(\mbox{SL}_3)$  then $V^a$ is a $K_{\nu}^{(3)}$-split unipotent algebraic subgroup of $\mbox{R}_{{K_\nu}/K_{\nu}^{(3)}}(\mbox{SL}_3)(K_{\nu}^{(3)})=G.$ We will try to avoid these and instead will keep our calculations explicit and down to earth. If $\{a,b,c\}$ is a basis for $K_\nu$ over $K_\nu^{(3)},$ then we let ${\rm pr}^{x}_{yz}$ (resp. ${\rm pr}^{xy}_{z}$) denote the projection onto the space $y\hh K_{\nu}^{(3)}+z\hh K_{\nu}^{(3)}$ (resp. $z\hh K_{\nu}^{(3)}$) parallel to $x\hh K_{\nu}^{(3)}$ (resp. $x\hh K_{\nu}^{(3)}+y\hh K_{\nu}^{(3)}$) where $x,y,z$ are distinct elements in $\{a,b,c\}.$  
For a basis $\{a,b,c\}$ of $K_{\nu}$ over $K_{\nu}^{(3)}$ define 
$$W^{b,c}_1=\left\{\left(\begin{array}{ccc}1& 0 & x\\ 0 & 1 & y\\ 0 & 0 & 1\end{array}\right):\h x\in bK_{\nu}^{(3)}+cK_{\nu}^{(3)},\h y\in K_{\nu}\right\}$$

Lemma~\ref{displacement} shows that we can find a sequence $\{g_n\}\subset G-DV$ such that $g_n\rightarrow e$ and $g_nY\cap Y\neq\emptyset.$ We construct a $V_1$-quasiregular map $\phi$ as in Lemma~\ref{gettingextrainv} using this sequence $\{g_n\}.$ The calculations following Lemma~\ref{gettingextrainv} show
\begin{equation}\label{e;char3}\phi(t)=d(\sigma(t))\h v_2\left(\frac{\vt(t)}{\sigma(t)^2}\right)\hspace{1cm}\end{equation} 
Furthermore thanks to Corollary~\ref{easycases} we may and will assume that $\sigma$ is constant map and $\vt(t)=at^3$ for some $a\neq0$ and $V_1V_2^aY=Y.$ Extend $\{a\}$ to a basis $\{a,b,c\}$ for $K_{\nu}$ over $K_{\nu}^{(3)}.$ 

Note that $Y$ is a (minimal) invariant set for $V^a.$ We apply Lemma~\ref{minimal2} with $M=\{g_n\}$ the same sequence as above, $P=V^a$ and $F=V_1.$ Hence $hY=Y$ for all $h\in N_G(V_1)\cap\overline{PMP}.$ 
Write as before
$$g_n=\left(\begin{array}{ccc}a_{1n}& a_{2n} & a_{3n}\\ b_{1n} & b_{2n} & b_{3n}\\ c_{1n} & c_{2n} & c_{3n}\end{array}\right)$$
A simple calculation shows that 
$$N_G(V^a)\subset \{g\in G:\h gV_1g^{-1}\subset V^a\}=DV$$ 
For any $t,s,a\in K_\nu$ let $v^a(t,s)=v_1(t)v_2(as^3).$ Since $\{g_n\}\subset G-DV$  and $N_G(V^a)\subset DV$ we get: $v^a(t,s)g_n(V^a)\subset G/V^a$ is a non-constant polynomial map of bounded degree into some affine space over $K_{\nu}^{(3)},$ here indeed we are using Chevalley's Theorem for the algebraic subgroup $V^a$ of $G$ as we mentioned above. Let $\mathfrak{B}^a$ be some relatively compact neighborhood of the coset $V^a$ in $G/{V^a}.$ Choose $r^a(n)$ such that $V^a_{r^a(n)}g_nV\not\subset\mathfrak{B}^a$ but $V^a_{k}g_nV\subset\mathfrak{B}^a$ for all $k<r^a(n).$ Denote by $\alpha_n^a$ the conjugation with $d(\theta^3)^{r^a(n)}$ and let $\sigma_n^a:V^a\rightarrow B^-,$ $\vt_n^a:V^a\rightarrow W_1^{b,c}$ and $\vt_n:V^a\rightarrow W_1$ be $K_{\nu}$-rational morphisms defined by
$$\alpha_n(v^a(t,s))g_n\beta_n^a(t,s)=\sigma_n^a(t,s)\vt_n^a(t,s)\h\mbox{and}\h\alpha_n(v^a(t,s))g_n\beta_n(t,s)=\sigma_n(t,s)\vt_n(t,s)$$ where $\beta_n^a(t,s)\in V^a$ and $\beta_n(t,s)\in V_1.$ In coordinates we have
$$\sigma_n^a(t,s)=\left(\begin{array}{ccc}\sigma_{11}^n(t,s)& 0 & 0\\ \sigma_{21}^n(t,s) & \sigma_{22}^n(t,s) & 0\\ \sigma_{31}^n(t,s) & \sigma_{32}^n(t,s) & \sigma_{33}^n(t,s)\end{array}\right)$$
$$\vt_n^a(t,s)=\left(\begin{array}{ccc}1& 0 &\vt_{13}^{an}(t,s)\\ 0 & 1 & \vt_{23}^{an}(t,s)\\ 0 & 0 & 1\end{array}\right)\h\mbox{and}\h\vt_n(t,s)=\left(\begin{array}{ccc}1& 0 &\vt_{13}^n(t,s)\\ 0 & 1 & \vt_{23}^n(t,s)\\ 0 & 0 & 1\end{array}\right)$$
Define 
$$\phi_n^a(t,s)=\alpha_n(v^a(t,s))g_n\beta_n^a(t,s)=\sigma_n^a(t,s)\vt_n^a(t,s)$$
Note that this construction fits into the same frame work as in Lemma~\ref{gettingextrainv}. Thus as in the proof of loc. cit. we may pass to the limit and  get a non-constant strongly $V^a$-quasiregular map $\phi^a(t,s)$\vspace{.5mm} such that  $\mbox{Im}(\phi^a)\subset N_G(V^a)\subset DV.$ 
We have 
\begin{itemize}
\item[($\dagger^a$)] $(\sigma_n^a)_{11}(t,s)=a_{1n}+b_{1n}\theta^{3r^a(n)}t+c_{1n}\theta^{6r^a(n)}\frac{t^2+2as^3}{2}=a_{1n}+{b^{'a}_{1n}}t+{c^{''a}_{1n}}\hh\frac{t^2+2as^3}{2}$
\item[($\ddagger^a$)] $\vt_{13}^{an}(t,s)=\frac{{\vt_n(t,s)}^a}{((\sigma_n^a)_{11}(t,s))^3},$ where ${\vt_n(t,s)}^a$ is a polynomial.
\end{itemize}

Passing to the limit we have 
$$\phi^a(t,s)=d(\sigma^a(t,s))\h v_2\left(\frac{\vt^a(t,s)}{(\sigma^a(t,s))^3}\right)$$
We need a more explicit description of ${\vt_n(t,s)}^a.$ Let us recall from $(\ddagger)$ that 
$$\vt_{13}^n(t)=\frac{2(v_1(\alpha_n(t))g_n)_{13}(v_1(\alpha_n(t))g_n)_{11}-((v_1(\alpha_n(t))g_n)_{12})^2}{2((v_1(\alpha_n(t))g_n)_{11})^2}=\frac{\vt_{0n}(t)}{(\sigma_n(t))^2}$$
are pre-limit functions for the construction of $\phi(t)$ in $\overline{V_1MV_1}.$ Similarly we have $\vt_{13}^n(t,s)=\frac{\vt_{0n}(t,s)}{(\sigma_n(t,s))^2}.$ If we apply the renormalization $\mathfrak{t}=\theta^{3r^a(n)}t$ and $\mathfrak{s}^3=\theta^{6r^a(n)}s^3,$ we may write
\begin{equation}\label{e;vt1}\vt_{0n}(\mathfrak{t})=A_{0n}+A_{1n}\mathfrak{t}+A_{3n}\mathfrak{t}^2+A'_{3n}\mathfrak{t}^2/2+A_{4n}\mathfrak{t}^3/2+A'_{4n}(\mathfrak{t}^2/2)^2\end{equation}
\begin{equation}\label{e;vt2}\vt_{0n}(\mathfrak{t},\mathfrak{s})=A_{0n}+A_{1n}\mathfrak{t}+A_{3n}\mathfrak{t}^2+A'_{3n}(\mathfrak{t}^2/2+a\mathfrak{s}^3)+A_{4n}\mathfrak{t}(\mathfrak{t}^2/2+a\mathfrak{s}^3)+A'_{4n}(\mathfrak{t}^2/2+a\mathfrak{s}^3)^2\end{equation}
In view of these formulas ${\vt_n(t,s)}^a$ is 
$${\rm pr}^{a}_{bc}\{\vt_{0n}(\mathfrak{t},\mathfrak{s})(a_{1n}+b^{'a}_{1n}\mathfrak{t}+c^{''a}_{1n}(\mathfrak{t}^2/2+a\mathfrak{s}^3))\}$$  

Recall that thanks to Corollary~\ref{easycases} in the equation~\eqref{e;char3} we have $\sigma(t)=1$ for all $t\in K_\nu$ and  $\vt(t)=at^3$ for some $a\neq0$. Thus $\vt(t)=\lim_n A_{4n}(\theta^{3r(n)}t)^3/2$ where $A_{4n}$ is as in~\eqref{e;vt1}.

We claim that there exists some constant $\bfrak>0,$ depending on $\mathfrak{B}$ chosen in the proof of Lemma~\ref{gettingextrainv} and $\mathfrak{B}^a$ which we chose here, such that $r^a(n)\leq \bfrak\hh r(n).$ To see the claim note that the above paragraph implies that $$\max_{|t|\leq1} A_{4n}(\theta^{3r(n)}t)^3/2\geq C=C(\bfr)\h\h\mbox{for all}\h\h n$$ 
thus $\max_{|t|,|s|\leq1} |A_{4n}\theta^{9r(n)}t(t^{2}/2+as^3)|\geq C$ for all $n.$ 
Also note that for $s\neq0$ the image of the polynomial 
$A_{4n}\theta^{9r(n)}t(t^{2}/2+as^3)),$ as a polynomial of $t,$ contains an open neighborhood of $0$ thus 
$$\lim_n{\rm pr}^{a}_{bc}\{a_{1n}(A_{4n}\theta^{9r(n)}t(t^{2}/2+as^3))\}$$ 
is a nontrivial polynomial. This implies the claim. 

From this claim and $\dagger^a$ we conclude that $\sigma^a(t,s)=1$ is the constant polynomial. Hence 
$$\phi^a(t)=v_2(\vt^a(t,s))\h\h\mbox{and}\h\h\vt^a(t,s)=\lim_n{\rm pr}^a_{bc}(a_{1n}\vt_{0n}(\mathfrak{t},\mathfrak{s}))$$ 
This in view of~\eqref{e;vt2} and the above discussion implies that 
$$\vt^a(t,s)=\lim_n{\rm pr}^a_{bc}\{a_{1n}(A'_{3n}a\mathfrak{s}^3+A_{4n}\mathfrak{t}(\mathfrak{t}^2/2+a\mathfrak{s}^3))\}$$
There are two possibilities: either $\lim_n{\rm pr}^a_{bc}\{a_{1n}(A_{4n}\mathfrak{t}(\mathfrak{t}^2/2+a\mathfrak{s}^3))\}=0$ is the constant polynomial or this limit is non-constant. Let us first assume the later holds. This in particular implies that $\lim_n A_{4n}\theta^{9r^a(n)}=A\neq0.$ Thus we have 
$$\lim_na_{1n}(A_{4n}\mathfrak{t}(\mathfrak{t}^2/2+a\mathfrak{s}^3))=A(t^{3}/2+ats^3)$$
On the other hands ${\rm Im}(A(t^{3}/2+ats^3))$ contains an open neighborhood of $0$ which then implies that $R=a K_\nu^{(3)}+{\rm Im}\vt^a(t,s)$ contains and open subset of $K_\nu.$ Besides $V_2(R)\hh Y=Y$ hence an argument similar to the proof of Proposition~\ref{cases} implies that (ii) in loc. cit. holds and Theorem~\ref{orbitclosure} follows from Corollary~\ref{easycases} in this case.

In view of the above we may and will assume that $\lim_n{\rm pr}^a_{bc}\{a_{1n}(A_{4n}\mathfrak{t}(\mathfrak{t}^2/2+a\mathfrak{s}^3))\}$ is constant. This is to say $\vt^a(t,s)=\lim_n{\rm pr}^a_{bc}(a_{1n}A'_{3n}a(\mathfrak{s}^3))=a's^3$ for some non-zero  $a'\in bK_{\nu}^{(3)}+cK_{\nu}^{(3)}.$ 

Thus $Y$ is invariant under $V^{aa'}=V_1V_2^aV_2^{a'}$ and is indeed $V_1$-minimal. We repeat the construction above one further time i.e. we construct $V^{aa'}$-quasi regular map. A simple calculation shows that $N_G(V^{aa'})\subset DV$ hence $\{g_n\}$ chosen above can be used again. We apply Lemma~\ref{minimal2} with $M=\{g_n\}$ the above sequence, $P=V^{aa'}$ and $F=V_1.$ In view of that lemma we have $hY=Y$ for any $h\in N_G(V_1)\cap\overline{PMP}.$ We use the notations as above with the obvious modification e.g. $\phi^a$ there is replaced by $\phi^{aa'}$ in here etc. Same argument as above gives: there exists some $\bfrak'>0$ such that $r^{aa'}(n)\leq \bfrak'\hh r(n).$ Hence $\sigma^{aa'}(t,s,r)=1$ is constant and  
$$\phi^{aa'}(t,s,r)=v_2(\vt^{aa'}(t,s,r))$$  

Let $\{a,a',a''\}$ be a basis for $K_\nu$ over $K_\nu^{(3)}.$ Define the renormalized variables $\mathfrak{t}=\theta^{3r^{aa'}(n)}t,$ $\mathfrak{s}^3=\theta^{6r^{aa'}(n)}s^3$ and $\mathfrak{r}^3=\theta^{6r^{aa'}(n)}r^3.$ We have
$$\vt^{aa'}(t,s,r)=\lim_n{\rm pr}^{aa'}_{a''}\{a_{1n}(A'_{3n}(a\mathfrak{s}^3+a'\mathfrak{r}^3)+A_{4n}\mathfrak{t}(\mathfrak{t}^2/2+a\mathfrak{s}^3+a'\mathfrak{r}^3))\}$$

Now either $\lim_n{\rm pr}^{aa'}_{a''}(a_{1n}A_{4n}\mathfrak{t}(\mathfrak{t}^2/2+a\mathfrak{s}^3+a'\mathfrak{r}^3)))=0$ is constant or a non-constant polynomial. If the later holds the argument goes through the same lines of the argument above. If the first holds then 
$$\vt^{aa'}(t,s,r)=\lim_n{\rm pr}^{aa'}_{a''}\{a_{1n}(A'_{3n}(a\mathfrak{s}^3+a'\mathfrak{r}^3))\}$$ 
Thus $\vt^{aa'}(t,s,r)=a''(l_1s^3+l_2r^3)$ where $l_i\in K_\nu^{(3)}$ and at least one of them is nonzero. Hence we get $Y$ is invariant under $V_2^{a''}$ and since $\{a,a',a''\}$ is a basis for $K_{\nu}$ over $K_{\nu}^{(3)}$ we have $V_2=V_2^aV_2^{a'}V_2^{a''}.$ This gives $V_2Y=Y$ i.e. we are in case (ii) of Proposition~\ref{cases}. Thus Corollary~\ref{easycases} finishes the proof.


\end{document}